\documentclass[12pt]{amsart}
\usepackage{amsmath, amssymb,graphicx,caption, subcaption, mathtools, enumitem}
\usepackage[pdftex,hyperindex=true]{hyperref}
\usepackage{url}
\usepackage{etoolbox}
\patchcmd{\section}{\scshape}{\bfseries}{}{}
\patchcmd{\subsection}{\bfseries}{\itshape}{}{}

\makeatletter
\def\@seccntformat#1{%
  \protect\textup{\protect\@secnumfont
    \ifnum\pdfstrcmp{section}{#1}=0 \bfseries\fi
    \ifnum\pdfstrcmp{subsection}{#1}=0 \itshape \fi
    \csname the#1\endcsname
    \protect\@secnumpunct
  }%
}  
\makeatother

\usepackage[utf8]{inputenc}
\usepackage[english]{babel}

\theoremstyle{plain}

\newtheorem{theorem}{Theorem}[section]

\newtheorem{corollary}[theorem]{Corollary}

\theoremstyle{definition}

\newtheoremstyle{note}
{3pt}
{3pt}
{\itshape}
{}
{\itshape}
{:}
{.5em}
{}

\theoremstyle{note}

\theoremstyle{plain} 
\newcommand{\thistheoremname}{}
\newtheorem*{genericthm}{\thistheoremname}

\numberwithin{equation}{section}

\newcommand{\acknowledge}{\subsection*{Acknowledgements}}

\def\Dbar{\leavevmode\lower.6ex\hbox to 0pt{\hskip-.23ex \accent"16\hss}D}

\begin{document}

\title{Determinants, Choices and Combinatorics}
\author{J. Malkoun}
\address{Department of Mathematics and Statistics\\
Faculty of Natural and Applied Sciences\\
Notre Dame University-Louaize, Zouk Mosbeh\\
P.O.Box: 72, Zouk Mikael,
Lebanon}
\email{joseph.malkoun@ndu.edu.lb}

\date{Received: date / Accepted: date}
\maketitle

\begin{abstract} We prove a formula which generalizes both Onn's colorful determinantal formula, related to Rota's basis conjecture, and Svrtan's $n!$ formula, related to 
the Atiyah-Sutcliffe problem. In some cases, our formula allows us to prove some results similar in spirit to the statement of Rota's basis conjecture. 
We prove such a result using Svrtan's $n!$ formula, generalizing one of Svrtan's arguments to a combinatorial setting.
\end{abstract}

\maketitle

\section{Introduction} \label{intro}

In 1989, Rota conjectured his famous basis conjecture which, at the time of writing, is still open in general.

\begin{theorem} Let $V$ be a finite-dimensional vector space over a field $\mathbb{F}$. Suppose, $\prescript{1}{}{A},\ldots,\prescript{n}{}{A}$ are $n$ 
bases of $V$. Then for each $i$, there is a linear order of $\prescript{i}{}{A}$, say 
\[\prescript{1}{}{A} = 
\{ \prescript{1}{}{a}_1, \prescript{1}{}{a}_2, \ldots, \prescript{1}{}{a}_n \}; \, \ldots \,; \prescript{n}{}{A} = \{ \prescript{n}{}{a}_1, \prescript{n}{}{a}_2, \ldots, \prescript{n}{}{a}_n  \}, \] 
such that 
\[ \prescript{}{1}{C} = \{ \prescript{1}{}{a}_1, \prescript{2}{}{a}_1, \ldots, \prescript{n}{}{a}_1 \}; \, \ldots \, ;\prescript{}{n}{C} = \{ \prescript{1}{}{a}_n, \prescript{2}{}{a}_n, \ldots, \prescript{n}{}{a}_n \}\] 
are $n$ bases.
\end{theorem}

The conjecture originated in Rota's work in invariant theory, and implies some identities in invariant 
theory and representation theory (cf. \cite{HR1994} and \cite{O1997}). The literature 
surrounding Rota's basis conjecture is vast, and shall not be reviewed here. The reader may look for instance at \cite{BD2015}, \cite{Chow1995}, 
\cite{CFGV2003}, \cite{C2009}, \cite{GH2006}, \cite{GW2007} and the references therein.

We do mention though S. Onn's article \cite{O1997}, where he proves his colorful determinantal formula, and uses it to show that the 
Alon-Tarsi conjecture on Latin squares implies Rota's basis conjecture for even $n$. The Alon-Tarsi conjecture on Latin squares states that if 
$n$ is even, then the number of even $n$-by-$n$ Latin squares is different from the number of odd $n$-by-$n$ Latin squares. 

The Atiyah--Sutcliffe problem on configurations of points (\cite{Ati-2000}, \cite{Ati-2001} and 
\cite{Ati-Sut-2002}) is a geometric problem with roots in Physics (\cite{BR1997}). Indeed, in (\cite{BR1997}), Berry and Robbins 
propose a quantum-mechanical and geometric approach to the spin-statistics theorem, to be contrasted with the standard argument 
which relies on quantum field theory. Their discussion there was mostly for $2$ or $3$ identical particles. In order to extend their work to $n$ 
identical particles, they needed a certain equivariant map to exist. This led to the so-called \emph{Berry-Robbins problem}. 
\\
\smallskip
\\
\textbf{Berry-Robbins Problem:} Does there exist, for each $n$, a continuous map 
\[ f_n: C_n(\mathbb{R}^3) \to U(n)/T^n \]
from the configuration space $C_n(\mathbb{R}^3)$ of $n$ distinct points in $\mathbb{R}^3$ into the complete flag manifold $U(n)/T^n$, such that 
$f_n$ is equivariant under the action of the symmetric group $\Sigma_n$? More precisely, $\Sigma_n$ acts on a configuration of $n$ distinct points 
in $\mathbb{R}^3$ by permuting the $n$ points of that configuration, while it acts by permuting the columns of an element of $U(n)$, and this 
action induces an action on the quotient manifold $U(n)/T^n$.
\\
\smallskip

In \cite{Ati-2000}, Atiyah solved positively the Berry-Robbins problem, but the maps $f_n$ in that article were not 
very satisfactory: they were not smooth for instance. This led Atiyah to consider other candidate maps for $f_n$, which in 
turn led to the Atiyah-Sutcliffe conjectures, the area being sometimes referred to as the Atiyah-Sutcliffe problem on configurations of points. 
Consider $n$ distinct points $x_i$ in $\mathbb{R}^3$, for $1 \leq i \leq n$. To each pair $(i,j)$, $1 \leq i,j \leq n$ and 
$i \neq j$, form the unit vector $v_{ij}$ in the direction from $x_i$ to $x_j$, and then let $p_{ij}$ be the linear polynomial 
with root $v_{ij}$, after identifying the unit sphere of directions in $\mathbb{R}^3$ with the Riemann sphere using 
stereographic projection. For each $i$ between $1$ and $n$, denote by $p_i$ the polynomial of degree less than $n$ 
defined by
\[
p_i = \prod_{j \neq i} p_{ij}
\]
The polynomials $p_i$ were conjectured by Atiyah in \cite{Ati-2000}, 
\cite{Ati-2001} to be linearly independent over $\mathbb{C}$, for any given configuration $(x_i)$ of $n$ distinct 
points in $\mathbb{R}^3$.  The conjecture 
was further refined by Atiyah together with Sutcliffe in \cite{Ati-Sut-2002}, where two 
successively stronger conjectures were also made. In 2018, linear independence was proved in \cite{AM2018}.

Svrtan in \cite{S2014} considered not only the Atiyah-Sutcliffe determinant, relevant to the Atiyah-Sutcliffe conjectures, but also 
some modified determinant functions, corresponding to various subgraphs of the complete unoriented graph on $n$ vertices, with 
the Atiyah-Sutcliffe determinant being one of them. There he proved his $n!$ formula, which in particular shows that, given any 
initial configuration of $n$ distinct points in $\mathbb{R}^3$, one of these 
modifed determinant functions must be non-zero.

In this article, we prove a formula (cf. Thm. \ref{main_thm}) in section 2 generalizing both Onn's colorful determinantal formula in \cite{O1997} and Svrtan's $n!$ formula in \cite{S2014}, and show 
how to obtain these two formulas as special cases in sections 3 and 4 respectively. In the process, we also generalize Svrtan's argument to a combinatorial setting (cf. Thm. \ref{second_thm}).

\section{An identity involving permutation groups and determinants} \label{main}

In this section, we introduce a general formula which includes as special cases Onn's colorful determinantal formula, as well as a formula by Svrtan, and more. Let $\mathbb{F}$ be 
a fixed field. All matrices in this section will be assumed to have entries in the ground field $\mathbb{F}$. Given a square matrix $A$ of dimension 
$n$-by-$n$, we denote by $[A]$ the $n$-tuple of its columns, namely $[A] = (a_1,\cdots,a_n)$, where $a_i$ is the $i$-th column of $A$. Let 
$\prescript{1}{}{A},\cdots,\prescript{k}{}{A}$ be $k$ square matrices of dimensions $n_1$-by-$n_1$,..., $n_k$-by-$n_k$, respectively. Given the $k$-tuple of square matrices 
$\mathcal{A} = (\prescript{1}{}{A},\ldots,\prescript{k}{}{A})$, we denote by
\[ [\mathcal{A}] = ([\prescript{1}{}{A}],\ldots,[\prescript{k}{}{A}]) \]
Let
\[ \Sigma = \Sigma_{n_1} \times \cdots \times \Sigma_{n_k} \] 
where $\Sigma_n$ is the symmetric group on $[n] = \{1,\ldots,n\}$. There is a natural action of $\Sigma$ on the space of all $k$-tuples of square matrices as above. Namely, if 
$\sigma \in \Sigma$, we write $\sigma = (\sigma_1,\ldots,\sigma_k)$ where $\sigma_i \in \Sigma_{n_i}$, and then define
\begin{equation} \sigma.\mathcal{A} =  (\sigma_1.(\prescript{1}{}{A}),\ldots,\sigma_k.(\prescript{k}{}{A})) \label{action} \end{equation}
where 
\[ (\rho.A)_{ij} = A_{i \rho^{-1}(j)} \qquad \text{($1 \leq i,j \leq n$)} \]
for any $n$-by-$n$ matrix $A$ and any $\rho \in \Sigma_n$. 
Thus we see that an element $\sigma \in \Sigma$ permutes the columns of each matrix 
$\prescript{i}{}{A}$, for $1 \leq i \leq k$.

Let $V^n$ denote the $n$-dimensional vector space $\mathbb{F}^n$ over $\mathbb{F}$, and denote by $V_n$ its dual vector space. Consider an element $f$ of the following tensor product space:
\begin{equation} f \in \bigotimes_{i=1}^k (V_{n_i})^{\otimes_{n_i}} \end{equation}
where $(V_n)^{\otimes_n}$ denotes the $n$-th tensor power of the vector space $V_n$.
We denote by $f(\mathcal{A}) \in \mathbb{F}$ the evaluation of the multilinear form $f$ at the ordered set of column vectors of $\mathcal{A}$ 
(strictly speaking, one should write $f([\mathcal{A}])$ instead, but we simplify the notation by dropping the square 
brackets). Often in applications, one is interested in an alternating sum of the form
\[ \sum_{\sigma \in \Sigma} \operatorname{sgn}(\sigma) f(\sigma^{-1}.\mathcal{A}) \]
where
\[ \operatorname{sgn}(\sigma) = \prod_{i=1}^k \operatorname{sgn}(\sigma_i) \]
We have the following theorem.
\begin{theorem} If $\mathcal{A} = (\prescript{1}{}{A},\ldots,\prescript{k}{}{A})$ is a $k$-tuple of square matrices and $f$ as above, then
\begin{equation} \sum_{\sigma \in \Sigma} \operatorname{sgn}(\sigma) f(\sigma^{-1}.\mathcal{A}) = \mathcal{I}(f,\mathbf{n}) \operatorname{det}(\prescript{1}{}{A})\cdots \operatorname{det}(\prescript{k}{}{A}) \label{formula}
\end{equation}
where $\operatorname{det}(A)$ denotes the determinant of a square matrix $A$, and $\mathcal{I}(f,\mathbf{n}) \in \mathbb{F}$ is a scalar depending only on $f$ and on the \emph{shape} 
$\mathbf{n} = (n_1,\ldots,n_k)$ of $\mathcal{A}$. \label{main_thm} 
\end{theorem}
\begin{corollary} As a corollary of the previous theorem, one can calculate $\mathcal{I}(f,\mathbf{n})$ by simply assuming that each $\prescript{i}{}{A}$ is an identity matrix, for $1 \leq i \leq k$. Namely, if we let
\[ I = (I_{n_1},\ldots,I_{n_k}) \]
where $I_{n_i}$ denotes the $n_i$-by-$n_i$ identity matrix, then
\begin{equation} \mathcal{I}(f,\mathbf{n}) = \sum_{\sigma \in \Sigma} \operatorname{sgn}(\sigma) f(\sigma^{-1}.I)\end{equation} \label{corollary}
\end{corollary}
In applications, $f$ is often a determinant, or a product of determinants (or sum thereof), which represents a combinatorial problem for which the 
input data $\mathcal{A}$ has ambiguity, where the 
possible choices for $\mathcal{A}$ are usually parametrized by the elements of $\Sigma$. The scalar $\mathcal{I}(f,\mathbf{n})$ represents a combinatorial invariant of the problem. 
If $\mathcal{I}(f,\mathbf{n})$ is non-zero, then for any input data $\mathcal{A}$ (with ambiguity) for which all matrices 
$\prescript{i}{}{A}$ are non singular for $1 \leq i \leq k$ (we will simply say in this case that $\mathcal{A}$ is non-singular), the right-hand-side of \eqref{formula} is non-zero, 
which implies that one of the terms in the alternating sum on the left-hand side of \eqref{formula} is non-zero. 
This translates into the statement that, if $\mathcal{I}(f,\mathbf{n}) \neq 0$, then given any non-singular input data $\mathcal{A}$ with ambiguity, a choice 
for $\mathcal{A}$ can be made so that some output matrix, or matrices, 
corresponding to $f$, are also non-singular. This is an outline of the general argument which will be made precise 
in the two applications we consider: Onn's colorful determinantal formula and Svrtan's $n!$ formula, and their consequences.

\begin{proof}[Proof of Thm. \ref{main_thm}]  The (partial) skew-symmetrization of $f$ belongs to the following $1$-dimensional space
\begin{equation} \sum_{\sigma \in \Sigma} \operatorname{sgn}(\sigma) (\sigma.f) \in \bigotimes_{i=1}^k (\Lambda^{n_i} V_{n_i}) \label{alternating} \end{equation}
where the action of $\Sigma$ on elements
\[ f \in \bigotimes_{i=1}^k (V_{n_i})^{\otimes_{n_i}} \] 
is given by $(\sigma.f)(\mathcal{A}) = f(\sigma^{-1}.\mathcal{A})$. But an element of $\Lambda^n V_n$ is a scalar multiple of the $n$-by-$n$ 
determinant, thought of as a multilinear form of the column vectors. We thus see, from equation \eqref{alternating}, that the alternating sum is some constant $\mathcal{I}(f,\mathbf{n})$ 
(constant meaning here independent of the values of the matrices $\prescript{i}{}{A}$) times 
the product of the determinants $\operatorname{det}(\prescript{i}{}{A})$, for $1 \leq i \leq k$. In other words, we have proved our main result, namely formula \eqref{formula}.
\end{proof}

\section{Onn's Colorful Determinantal Formula}
In this section, we assume that the ground field $\mathbb{F}$ is a field of characteristic $0$, and that $\mathcal{A} = (\prescript{1}{}{A}, \ldots, \prescript{n}{}{A})$, 
where each $\prescript{i}{}{A}$ is an $n$-by-$n$ matrix for $1 \leq i \leq n$. We denote by 
$\prescript{i}{}{A}^j$ the $j$-th column of the matrix $\prescript{i}{}{A}$. We define the multilinear form $f$ by
\[ f(\mathcal{A}) = \prod_{i=1}^n \operatorname{det}(\prescript{1}{}{A}^i, \ldots, \prescript{n}{}{A}^i) \]
We introduce the group
\[ \Sigma = \Sigma_n \times \cdots \times \Sigma_n \qquad \text{(n factors)} \]
where $\Sigma_n$ is the symmetric group on $[n]$. We write an arbitrary element $\sigma \in \Sigma$ as $\sigma = (\sigma_1,\ldots,\sigma_n)$.
We then form the alternating sum
\begin{equation} \sum_{\sigma \in \Sigma} \operatorname{sgn}(\sigma) f(\sigma^{-1}.\mathcal{A}) = 
\sum_{\sigma \in \Sigma} \operatorname{sgn}(\sigma) \prod_{j=1}^n \operatorname{det}(\prescript{1}{}{A}^{\sigma_1(j)}, \ldots, \prescript{n}{}{A}^{\sigma_n(j)}) \label{alternating_Onn} \end{equation}
By Thm. \ref{main_thm}, we obtain that
\begin{equation} \sum_{\sigma \in \Sigma} \operatorname{sgn}(\sigma) f(\sigma^{-1}.\mathcal{A}) = \mathcal{I}(f,n) \operatorname{det}(\prescript{1}{}{A}) \cdots \operatorname{det}(\prescript{n}{}{A}) \label{pre_Onn} \end{equation}
It remains to calculate $\mathcal{I}(f,n)$ using Cor. \ref{corollary}. We thus assume that
\[ \prescript{i}{}{A} = I_n \qquad \text{($1 \leq i \leq n$)}\]
where $I_n$ is the $n$-by-$n$ identity matrix. It can be shown that (cf. for instance \cite{Zappa} or \cite{AL})
\[ \mathcal{I}(f,n) = \sum_{\sigma \in \Sigma} \operatorname{sgn}(\sigma) \prod_{j=1}^n \operatorname{det}(e_{\sigma_1(j)},\ldots,e_{\sigma_n(j)}) = l(n) \]
where $l(n)$ is the number of even $n$-by-$n$ Latin squares minus the number of odd $n$-by-$n$ Latin squares. Combining the previous formula with formulas \eqref{alternating_Onn} 
and \eqref{pre_Onn}, we obtain Onn's colorful determinantal identity:
\begin{equation} \sum_{\sigma \in \Sigma} \operatorname{sgn}(\sigma) \prod_{j=1}^n \operatorname{det}(\prescript{1}{}{A}^{\sigma_1(j)}, \ldots, \prescript{n}{}{A}^{\sigma_n(j)}) 
 = l(n) \prod_{i=1}^n \operatorname{det}(\prescript{i}{}{A}) \label{Onn} \end{equation}
 The Alon-Tarsi conjecture is that $l(n) \neq 0$ if $n$ is even. Onn used his formula to show that the Alon-Tarsi conjecture implies Rota's basis conjecture for even $n$, 
 which had been proved by Huang and Rota in \cite{HR1994}. The argument is as 
 follows. Assume $n$ is even and that the Alon-Tarsi conjecture is true, i.e. that $l(n) \neq 0$. If the matrices $\prescript{i}{}{A}$ are non-singular for all $i$, $1 \leq i \leq n$, then the right-hand side of 
 Onn's formula \eqref{Onn} is non-zero, which implies that at least one of the terms in the alternating sum on the left-hand side of the same formula must be non-zero. But this is precisely the conclusion of 
 Rota's basis conjecture.

\section{Svrtan's $n!$ formula}
In this section, we choose to work over $\mathbb{C}$, though the results will all hold over any field $\mathbb{F}$ of characteristic $0$ (or more generally 
any field in which $n!$ is non-zero). Inspired by both Rota's basis conjecture, as well as the Atiyah-Sutcliffe problem on configurations of 
points, the author came up with a hybrid problem, which he initially posted on an online forum (Math 
StackExchange at first, but then migrated to MathOverflow). The problem can be described as follows.

For $m$ a positive integer, denote by $V_m$ the space of complex polynomials of degree less than $m$ in one complex variable. 
Let $\Gamma_n$ be the complete oriented graph on $n$ vertices. The set of vertices may be taken to be 
$\{1,\ldots,n\}$, and the set of oriented edges becomes then 
$\{(i,j); 1\leq i,j \leq n, i \neq j\}$. We will denote the oriented edge $(i,j)$ simply by $e_{ij}$. One may 
then introduce an equivalence relation $\sim$, where each $e_{ij}$ is equivalent 
only to itself and to $e_{ji}$. We denote the equivalence class of $e_{ij}$ by $[e_{ij}]$, and the quotient 
$\overline{\Gamma}_n$ of $\Gamma_n$ by $\sim$ is the non-oriented complete 
graph on $n$ vertices, whose set of (non-oriented) edges is $\{ [e_{ij}]; 1 \leq i < j \leq n \}$.

\begin{figure}
\begin{center}
\includegraphics{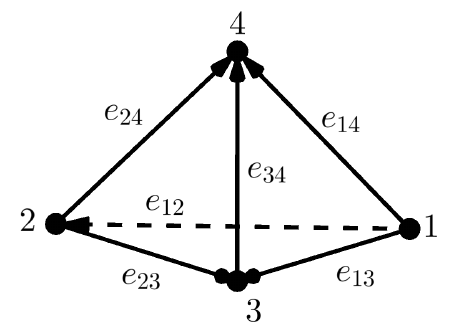}
\caption{A representation of the graph $\Gamma_4$}
\end{center}
\end{figure}
Assume there given, to each non-oriented edge $[e_{ij}]$, $i<j$, of $\overline{\Gamma}_n$, an 
ordered basis $(\prescript{ij}{}p_1, \prescript{ij}{}p_2)$ of the space $V_2$. We call such an assignment an assignment of 
``spinor bases''. We refer to $V_2$ as the space of spinors. 

By an assignment $c$ of spinors corresponding to an assignment of spinor bases, we mean that to each $e_{ij}$, with $i < j$, we assign either 
$\prescript{ij}{}p_1$ or $\prescript{ij}{}p_2$, and we assume that once such a choice is made, the ``other'' choice is made for the edge $e_{ji}$ ($i < j$). 
So for instance, if to $e_{ij}$ we assigned $\prescript{ij}{}p_2$, we must then assign $\prescript{ij}{}p_1$ to $e_{ji}$. Once an assignment of spinors is made, 
we denote the spinor associated to each edge $e_{ij}$ ($i \neq j$) simply by $p^c_{ij}$, where $c$ refers to the assignment of spinors.

We denote by $\mathcal{C}$ the set of all possible assignments of spinors (corresponding to some fixed assignment of spinor bases), which contains 
exactly $2^{\binom{n}{2}}$ elements. We define the sign $\operatorname{sgn}(c)$ of an assignment $c \in \mathcal{C}$ of spinors to be $1$ (resp. $-1$) if the 
number of times $\prescript{ij}{}p_2$ was assigned to the edge $e_{ij}$ for $i<j$ is even (resp. odd). We say that an assignment of spinors $c$ is even 
(resp. odd) if $\operatorname{sgn}(c) = 1$ (resp. $\operatorname{sgn}(c) = -1$).

To each assignment $c \in \mathcal{C}$ of spinors, one forms $n$ complex polynomials $p^c_i \in V_n$, for $1 \leq i \leq n$, of degree less than $n$ 
in the complex variable $t$, obtained as follows. The polynomial $p_i$ is defined as:
\[ p^c_i = \prod_{j \neq i} p^c_{ij} \qquad \text{($1 \leq i,j \leq n$)} \]
The problem can now be stated.
\\
\smallskip
\\
\underline{Problem:} Given an arbitrary inital assignment of spinor bases $(\prescript{ij}{}p_1, \prescript{ij}{}p_2)$ to each non-oriented 
edge $[e_{ij}]$, $i<j$, of the (non-oriented) complete graph $\overline{\Gamma}_n$ on $n$ vertices,
does there always exist an associated assignment $c$ of spinors to each oriented edge $e_{ij}$, $i \neq j$, of the oriented complete graph $\Gamma_n$, 
for which the $n$ polynomials $p^c_i$, for $1 \leq i \leq n$, are linearly independent over $\mathbb{C}$?
\\
\smallskip
\begin{theorem} \label{second_thm}The answer to the previous problem is positive, namely, given any initial assignment 
of spinor bases to each (non-oriented) edge of $\overline{\Gamma}_n$, there always exists an associated assignment 
$c$ of spinors to each (oriented) edge of $\Gamma_n$, for which the associated $n$ polynomials 
$p^c_i$, for $1 \leq i \leq n$, are linearly independent over $\mathbb{C}$. \end{theorem}
\smallskip

Svrtan's theorem \cite[pp. 21--23]{S2014}, slightly modified to be adapted to our notation, can be written as follows:
\begin{theorem}[Svrtan's $n!$ Formula]
\begin{equation} \label{Svrtan} \sum_{c \in \mathcal{C}} \operatorname{sgn}(c) \operatorname{det}(p^c_1,\ldots,p^c_n) = n! \prod_{1 \leq i<j \leq n} \operatorname{det}(\prescript{ij}{}p_1, \prescript{ij}{}p_2) \end{equation}
\end{theorem}
\begin{proof}[Proof of Svrtan's Theorem] The proof follows from Thm. \ref{main_thm}, with $\mathcal{A} = ((\prescript{ij}{}{p}_1, \prescript{ij}{}{p}_2); 1 \leq i<j \leq n)$, and 
\[ f(\mathcal{A}) = \operatorname{det}(p^{c_0}_1,\ldots,p^{c_0}_n) \]
where $c_0 \in \mathcal{C}$ corresponds to assigning $\prescript{ij}{}{p}_1$ to each $e_{ij}$, $1 \leq i<j \leq n$ ($c_0$ is thus even). Indeed formula \eqref{formula} then implies
\begin{equation} \sum_{c \in \mathcal{C}} \operatorname{sgn}(c) \operatorname{det}(p^c_1,\ldots,p^c_n) = \mathcal{I}(f,n) \prod_{1 \leq i<j \leq n} \operatorname{det}(\prescript{ij}{}p_1, \prescript{ij}{}p_2) 
\label{pre_Svrtan} \end{equation}
It remains only to check that $\mathcal{I}(f,n) = n!$. In order to calculate 
$\mathcal{I}(f,n)$, it suffices by Cor. \ref{corollary} to assume that
\begin{equation} \left\{ \begin{array}{cc} \prescript{ij}{}{p}_1 &= 1 \\
\prescript{ij}{}{p}_2 &= t \end{array} \right. \label{special} \end{equation}
($1 \leq i < j \leq n$). Given any (non-oriented) edge $[e_{ij}]$ ($i<j$) of $\overline{\Gamma}_n$, we orient it (i.e. assign to it a direction) from $i$ to $j$ 
(resp. from $j$ to $i$) if $\prescript{ij}{}{p}_1 = 1$ is assigned to $e_{ij}$ ($\prescript{ij}{}{p}_2 = t$ is assigned to $e_{ij}$). We see that there must be exactly 
one vertex having exactly $r$ outgoing vertices, for every $r$ with $0 \leq r \leq n-1$. Therefore only $n!$ terms are non-zero in the alternating sum on the left-hand side of 
\eqref{Svrtan}, the ones corresponding to some complete order $\prec$ on the set $[n] = \{1,\ldots,n\}$, with the corresponding assignment of spinors being
\[ i \prec j \implies \prescript{ij}{}{p} = 1 \]
More precisely, $\det(p^{c}_1,\ldots,p^{c}_n) = 1$ (resp. $-1$) if $c$ corresponds to a complete order $\prec$ which in turn corresponds to an even (resp. odd) permutation of $[n]$. 
We therefore deduce that
\[ \mathcal{I}(f,n) = \sum_{c \in \mathcal{C}} \operatorname{sgn}(c) \operatorname{det}(p^c_1,\ldots,p^c_n) = n! \]
corresponding to the assignment of spinor bases in \eqref{special}. Svrtan's formula then follows, using equation \eqref{pre_Svrtan}.
\end{proof}

We can now prove Thm. \ref{second_thm}. Indeed, given an assignment of spinor bases, it is clear that the right-hand side of Svrtan's $n!$ formula \eqref{Svrtan} is non-zero, so that 
one of the terms in the alternating sum on the left-hand side of the same formula must be non-zero. This implies that there is some associated assignment $c \in \mathcal{C}$ of spinors, 
for which the polynomials $p^c_1,\ldots,p^c_n$ are linearly independent over $\mathbb{C}$, proving the theorem.

\acknowledge{The author thanks Timothy Chow for an interesting email exchange
in which Dr. Chow informed the author about some of the 
literature surrounding Rota's basis conjecture, and suggested a line of attack similar to one 
initiated by Rota himself and applied to his conjecture, which had led 
to a link between his conjecture and the Alon-Tarsi conjecture concerning the number 
of even/odd latin squares of order $n$ (cf. \cite{HR1994} for the original work, 
and \cite{O1997} for a nicely written alternative short argument). The author is also indebted to the anonymous referee, who 
in particular provided a simpler proof of Svrtan's $n!$ formula than the author's original 
proof, and this inspired the author to find a more general formula, using a similar idea.}
\smallskip

\vspace{5mm}

\def\Dbar{\leavevmode\lower.6ex\hbox to 0pt{\hskip-.23ex \accent"16\hss}D}
\providecommand{\bysame}{\leavevmode\hbox to3em{\hrulefill}\thinspace}
\providecommand{\MR}{\relax\ifhmode\unskip\space\fi MR }
\providecommand{\MRhref}[2]{%
  \href{http://www.ams.org/mathscinet-getitem?mr=#1}{#2}
}
\providecommand{\href}[2]{#2}

\end{document}